\documentclass[11pt]{article}
\usepackage{amssymb,amsmath}
 \usepackage{fullpage}
\usepackage{tikz}
\usepackage[left,mathlines,modulo]{lineno}

\newtheorem{theorem}{\bf Theorem}[section]
\newtheorem{corollary}[theorem]{\bf Corollary}
\newtheorem{lemma}[theorem]{\bf Lemma}
\newtheorem{proposition}[theorem]{\bf Proposition}

\newcommand{\qed}{\hfill $\square$ \bigskip}

\newcommand{\Fib}{{\cal F}}

\newcommand{\R}{\mathcal{R}}

\newcommand{\M}{\mathcal{M}}
\newcommand{\B}{{\mathcal B}}
\newcommand{\V}{{\mathcal V}}
\newcommand{\A}{{\mathcal A}}

\newcommand{\st}{ ~|~ }

\numberwithin{equation}{section}

\begin{document}
\modulolinenumbers[5]
\title{Distance cube polynomials of Fibonacci and Lucas-run graphs}

\author{
Michel Mollard\footnote{Institut Fourier, CNRS, UMR 5582, Universit\'e Grenoble Alpes, CS 40700, 38058 Grenoble Cedex 9, France}
}
\date{\today}
\maketitle

\begin{abstract}
The Fibonacci-run graphs $\R_n$ are a familly of  an induced subgraph of hypercubes
introduced by E\u{g}ecio\u{g}lu and Ir\v{s}i\v{c} in 2021.
A cyclic version of $\R_n$, the  Lucas-run graph  $\R_n^l$,  was also recently proposed (Jianxin Wei, 2024).
We prove that the generating function previously given for the polynomial $D_{\R_n}(x,q)$ which counts the number of hypercubes at a given distance in $\R_n$ was erroneous and determine its correct expression. We also consider Lucas-run graphs and prove the conjecture proposed by Jianxin Wei establishing the link between cube polynomials of $\R_n^l$ and $\R_n$.

\end{abstract}

%

\noindent
{\bf Keywords:} Fibonacci cubes, Lucas cubes, Fibonacci-run graphs, Lucas-run graphs, Hypercube

\noindent
{\bf AMS Subj. Class. }:  05C12, 05C30, 05C31, 05C60, 05C90

\section{Introduction}
\label{sec:Introduction}

The {\em Fibonacci cube} of dimension $n$, denoted as $\Gamma_n$, is the subgraph of the hypercube $Q_n$ induced by vertices with no consecutive 1s. This graph was introduced in \cite{H-1993a} as an interconnection network.

$\Gamma_n$ is an isometric subgraph and is inspired by the Fibonacci numbers. It has attractive recurrent structures such as
its decomposition into two subgraphs which are also Fibonacci cubes themselves.  Structural properties of these graphs were more extensively 
studied afterwards; see for example the survey~\cite{K-2013a} and the recent book~\cite{EKM-2023}. \\
\indent Lucas cubes $\Lambda_n$, introduced in \cite{MPZ-2001}, have attracted attention as well due to the fact that these cubes are 
the cyclic version of Fibonacci cubes. They have also been widely studied.

Not only the investigation of the 
properties of Fibonacci cubes and Lucas cubes attracted many researchers, but it
has also led to the development of a variety of interesting 
generalizations and variations covered by a whole chapter in the book
\cite{EKM-2023}. Among these families of graphs are 
generalized Fibonacci cubes, Pell graphs, 
$k$-Fibonacci cubes, daisy cubes, Fibonacci $p$-cubes and Fibonacci-run graphs.

Fibonacci-run graphs $\R_n$ are subgraphs of hypercubes induced by binary strings with restricted run lengths. The basic properties of Fibonacci-run graphs are studied in~\cite{EI-2021a} and its companion paper~\cite{EI-2021b}. 

Among these properties the counting polynomial $D_{\R_n}(x,q)$ of the number of induced subgraphs of $\R_n$ isomorphic to $Q_k$ at a given distance  of $0^n$ is studied in~\cite{EI-2021a}. This polynomial is a generalization of  $C_{\R_n}(x)=\sum_{k\geq0}c_{k}x^k$ where $c_{k}$ is the number of subgraphs of $\R_n$ isomorphic to $Q_k$.  We prove in Section~\ref{sec:FR} that the generating function previously given for $D_{\R_n}(x,q)$ is erroneous and determine its correct expression.

A cyclic version of $\R_n$, the Lucas-run graph $\R_n^l$ has been introduced and studied in~\cite{W-2024a}. Inspired by relations $|V(\R_n^l)|=2|V(\R_{n-1})|-|V(\R_{n-2})|$ and $|E(\R_n^l)|=2|E(\R_{n-1})|-|E(\R_{n-2})|$ it is conjectured in~\cite{W-2024a} that $C_{R_n^l}(x)=2C_{R_{n-1}}(x)-C_{R_{n-2}}(x).$ In Section \ref{sec:LR} we prove the stronger result $D_{R_n^l}(x,q)=2D_{R_{n-1}}(x,q)-D_{R_{n-2}}(x,q).$

The last section is a conclusion consisting of the presentation of two open problems and a quick discussion on the application of the method used in this article to Fibonacci cubes and Lucas cubes.

In the next section concepts and results 
needed are given.
\section{Preliminaries}
\bigskip

We denote by $[n]$ the set of integers $i$ such that $1\leq i \leq n$.

Throughout this article the use of the union symbol $A\cup B$ or $\bigcup_{k\geq 0}{A_k}$ assumes that the sets are disjoint.

Let $\{F_n\}$ be the \emph{Fibonacci numbers}:
$F_0 = 0$, $F_1=1$, $F_{n} = F_{n-1} + F_{n-2}$ for $n\geq 2$.

Let $B=\{0,1\}$ and $\B_n$ be the set of strings of length $n$ over $B$
$$
\B_n= \{u_1u_2\ldots u_n \st u_i \in B \}\,.
$$
the {\em Hamming weight} of $u\in \B_n$ is $w(u) = \sum_{i=1}^n u_i$, that is, the number of bits $1$ in $u$. We note $l(u)$ the length of a binary string $u$.  We will also use power notation for the concatenation of similar bits, for instance $0^n = 0\ldots 0\in \B_n$.

Given a set $F$ of binary strings we will denote by $0F$, $F0$, $F00$ and $0F0$ the sets of strings $0F=\{0s\st s\in F\}$, $F0=\{s0\st s\in F\}$, $F00=\{s00\st s\in F\}$ and $0F0=\{0s0\st s\in F\}$.

 The {\em $n$-dimensional hypercube} $Q_n$ is the graph with strings of $\B_n$ as vertices, vertices $u_1\ldots u_n$ and $v_1\ldots v_n$ being adjacent if $u_i\ne v_i$ for exactly one $i\in [n]$. 

The \emph{distance} $d_G(u,v)$ between two vertices $u$ and $v$ of a graph $G$  is the 
number of edges on a shortest shortest $u,v$-path. It is immediate that the distance between two vertices of $Q_n$ is  the number of bits the strings differ, sometime called Hamming distance $H(u,v)$. Note that $d_{Q_n}(0^n,v)=w(v)$ for any vertex $v$ of $Q_n$. 

Let $G$ be an induced subgraph of $Q_n$. Since each neighbor of $v \in V(G)$ is obtained either by changing a $0$ in $v$ into a $1$ 
or by changing a $1$ in $v$ into a $0$, we can distinguish between the 
\emph{up-degree} $\deg_{\rm up} (v)$, and \emph{down-degree} $\deg_{\rm down} (v)$ of the vertex $v$. 
The first is the number of up-neighbors of $v$ 
while the second is the number of down-neighbors of $v$.

The following result is well-known.

\begin{proposition}\label{pro:bt}
In every induced subgraph $H$ of $Q_n$ isomorphic to $Q_k$ there exists a unique vertex of minimal Hamming weight, \emph{the bottom vertex} $b(H)$. There exists also a unique vertex of maximal Hamming weight, the \emph{top vertex} $t(H)$. 
Furthermore $b = b(H)$ and $t = t(H)$ are at distance $k$ and characterize $H$ among the subgraphs of $Q_n$ isomorphic to $Q_k$. 
\end{proposition}

Let $G$ be an induced subgraph of $Q_n$. If $H$ is an induced subgraph of $G$, it is also an induced subgraph of $Q_n$. Thus Proposition~\ref{pro:bt} is still true for induced subgraphs of $G$. We call $w(b(H))$ the \emph{distance of $H$}.

For a graph $G$, let $c_k$ $(k\ge 0)$ be the number of induced 
subgraphs of $G$ isomorphic to  $Q_k$. The {\em cube polynomial}, $C_G(x)$,
of $G$, is the corresponding enumerator polynomial, that is

\begin{equation}\label{eqn:defCG}
C_G(x) = \sum_{k\geq 0} c_k x^k\,.
\end{equation}
This polynomial was introduced in~\cite{BKS-2003}.

Assume $0^n$ belongs to $G$ subgraph of $Q_n$. A generalization of $C_{G}(x)$ is the {\em distance cube polynomial} with respect to $0^n$~\cite{KM-2019a} first introduced as {\em$q$-cube polynomial}{~\cite{SE-2017a,SE-2018a} in the case of graphs where it can be seen as a $q$-analog of the cube polynomial. We define this polynomial the following way

\begin{equation}\label{eqn:defCGq}
D_G(x,q) = \sum_{k\geq 0} c_{k,d} x^kq^d 
\end{equation} where $c_{k,d}$ is  the number of induced 
subgraphs of $G$ isomorphic to $Q_k$ with a bottom vertex at distance $d$ of $0^n$.

Let $\A$ be a set of strings generated freely (as a monoid) by a finite or  infinite alphabet
$E$. 
This means that every strings $s \in \A$ can be written 
uniquely as a concatenation of zero or more strings from $E$. Note that the empty string $\lambda$ belongs to $\A$.
Associate to any string $s \in \A$ a polynomial, possibly multivariate,  $\theta_s$. We will say that $s \rightarrow \theta_s$ is \emph{concatenation multiplicative} over $\A$ if $\theta_{\lambda}=1$ and $\theta_{e_1e_2\ldots e_k}=\theta_{e_1}\theta_{e_2}\ldots\theta_{e_k}$ for any choice of elements $e_i$ in $E$.

 For example the monoid of finite binary strings $\bigcup_{n\geq 0}{\B_n}$ is generated by $B=\{0,1\}$. On this monoid  $\theta_s(x)=x^{w(s)}$ is concatenation multiplicative. On the other hand if $m(s)$ denote the maximum number of consecutive $1$s in $s$ then $\theta_s(x)=x^{m(s)}$ is not concatenation multiplicative since $\theta_{1101}(x)=x^2\neq \theta_1(x)\theta_1(x)\theta_0(x)\theta_1(x)=x^3$.

Note that if $s \rightarrow \theta_s$ and $s \rightarrow \phi_s$ are concatenation multiplicative then $s \rightarrow \theta_s\phi_s$ is concatenation multiplicative. 

\begin{proposition}\label{pro:mono}

Let $\A$ be a set of strings generated freely, as a monoid, by the alphabet $E$ and $s \rightarrow \theta_s$ concatenation multiplicative over $\A$ then
\begin{equation*}
\sum_{s\in \A}{\theta_s}=\frac{1}{1-\sum_{e\in E}{\theta_e}}\,.
\end{equation*}
\end{proposition}
\begin{proof}
For any integer $k$ let $\A_k$ be the set of strings that are concatenation of $k$ strings of $E$ then since $\A= \bigcup_{k\geq 0}{\A_k}$
\begin{equation*}
\sum_{s\in \A}{\theta_s}=\sum_{k\geq0}{\sum_{s\in \A_k}{\theta_s}}\,.
\end{equation*}
On the other hand
\begin{equation*}
\frac{1}{1-\sum_{e\in E}{\theta_e}}=\sum_{k\geq0}{(\sum_{e\in E}{\theta_e})^k}
\end{equation*}
and for any integer $k$
\begin{equation*}
(\sum_{e\in E}{\theta_e})^k=\sum_{(e_1,e_2,\ldots, e_k)\in E^k}{\theta_{e_1}\theta_{e_2}\ldots\theta_{e_k}}
=\sum_{(e_1,e_2,\ldots e_k)\in E^k}{\theta_{e_1e_2\ldots e_k}}=\sum_{s\in \A_k}{\theta_s}\,.
\end{equation*}
\end{proof}\qed

A \emph {Fibonacci string} is a binary strings without consecutive $1$s.

The \emph {Fibonacci cube} $\Gamma_n$  is the subgraph of $Q_n$ induced by Fibonacci strings of length $n$, thus for $n\geq2$:
\begin{eqnarray*}
V(\Gamma_n) &  = &
\{u_1u_2 \ldots u_n \st \forall i \in [n-1]\,u_i.u_{i+1}=0\} \, , \\
E(\Gamma_n) & = & 
\{ \{ u, v \} \st H (u,v) = 1 \} \, .
\end{eqnarray*}


According to the definition  $V(\Gamma_0)=\{\lambda\}$ and $V(\Gamma_1)=\{0,1\}$. Note that $|V(\Gamma_n)|=F_{n+2}$.

The \emph {Lucas cube} $\Lambda_n$ is the subgraph of $Q_n$ induced by strings without consecutive $1$s circularly:
\begin{eqnarray*}
V(\Lambda_n) &  = &
\{u_1u_2 \ldots u_n \st \forall i \in [n-1]\,u_i.u_{i+1}=0 \,and\,u_n.u_1=0\} \, . \\
\end{eqnarray*}

We set  $V(\Lambda_0)=\{\lambda\}$ and $V(\Lambda_1)=\{0\}$.

As recalled in~\cite{EI-2021a}
   extended Fibonacci strings, obtained by adding 00 to the Fibonacci strings, together with 
the null string $\lambda$ and the singleton $0$ are generated freely as a monoid by the infinite alphabet 

\begin{equation}\label{eqn:F}
F = \{ 0,100, 10100,1010100, \ldots  \}
\end{equation} 

We note $\Fib=\{\lambda,0\} \cup(\bigcup_{n\geq 0}{V(\Gamma_n)00})$ this monoid.

A run in a binary string is a maximal length substring of  $0$s or $1$s.

We consider now \emph{run-constrained binary strings}
which are strings of $0$s and $1$s
in which every run of $1$s is immediately followed by a strictly longer run of $0$s.
Such run-constrained strings, including null-word $\lambda$ and the singleton $0$, 
are the elements of the monoid $\M$ generated freely by the letters from the alphabet
\begin{equation}\label{eqn:R}
R = \{ 0,100, 11000,1110000, \ldots  \}
\end{equation}
Note that run-constrained strings of length $ n \geq 2$ must end with $00$. Let $\V_n$ be the set of run-constrained strings of length $n$. 

The \emph{Fibonacci-run graph} $\R_n$ is the induced subgraph of $Q_n$ defined as follows:
\begin{eqnarray*}
V(\R_n) &  = &
\{ s \st s00 \in \V_{n+2} \} \, , \\
E(\R_n) & = & 
\{ \{ u, v \} \st H (u,v) = 1 \} \, .
\end{eqnarray*}

We have thus $\M=\bigcup_{k\geq 0}{\V_k}=\{\lambda,0\}\cup(\bigcup_{n\geq 0}{V(\R_n)}00)$.

There is a bijection between extended Fibonacci strings and run-constrained strings. Indeed define  $\Phi:F \rightarrow R$ by setting $\Phi((10)^i 0) = 1^i 0^{i+1} $ for $ i \geq 0 $, and then
extend $\Phi$ to full words via the unique factorization.

A consequence of this bijection is that $|V(\R_n)|=|V(\Gamma_n)|=F_{n+2}$.

Following \cite{W-2024a} a \emph{circular-run-constrained string} is defined as a string where every run of $1$s is immediately followed by a strictly longer run of 0s circularly. Let $\V_n^c$ be the set of circular-run-constrained strings of length $n$.

For $n\geq 0$, the \emph{Lucas-run graph} $\R_n^l$ is subgraph of $Q_n$ induced by the set of vertices

\begin{equation*}
V(\R_n^l)  =
\{ s \st s0 \in \V_{n+1}^c \text{ and } s00 \in \V_{n+2} \} \, . 
\end{equation*}

 Note that $\R_n^l$ is an induced subgraph of $\R_n$.

The definition of Lucas-run graph is motivated by the fact that $\bigcup_{n\geq 0}\{V(\R_n^l)00\}$ is the image of $\bigcup_{n\geq 0}\{V(\Lambda_n)00\}$ by the bijection $\Phi$ and thus  $|V(\R_n^l)|=|V(\Lambda_n)|$.

Another cyclic variation of Fibonacci-run graphs has been proposed in~\cite{WY-2024}

\section{Some complements about Fibonacci-run graphs}\label{sec:FR}
\bigskip

There is obviously something wrong in the value of the generating function of the $q$-cube polynomials $D_{\R_n}(x,q)$ given in~\cite[Proposition~8.2]{EI-2021a}. Indeed from this generating function the polynomial for $n=5$ should be
\begin{equation*} 1+5q+7q^2+(5+14q)x+7q^2
\end{equation*}
and by direct inspection of Fig.~\ref{fig:R_5} this polynomial is
 \begin{equation*}
D_{\R_5}= 1+5q+6q^2+q^3+(5+12q+2q^2)x+(6+q)x^2.
\end{equation*}

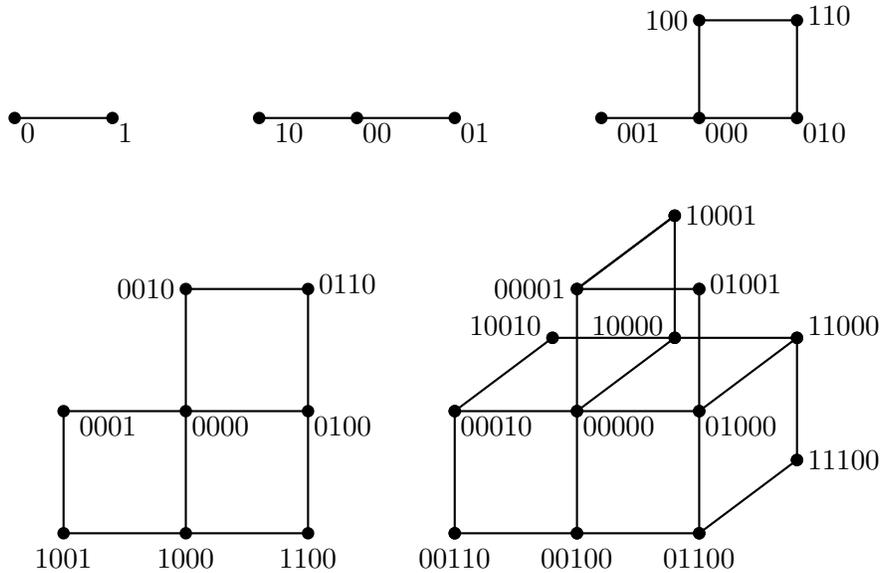
\begin{figure}[!ht]
\begin{center}
\begin{tikzpicture}[scale=0.65,style=thick]
\tikzstyle{every node}=[draw=none,fill=none]
\def\vr{3pt} 
\begin{scope}[xshift =0, yshift = 6cm]

\path (7,2.5) coordinate (v001);
\path (5,2.5) coordinate (v000);
\path (3,2.5) coordinate (v010);
\path (7,4.5) coordinate (v101);
\path (5,4.5) coordinate (v100);

\path (0,2.5) coordinate (v01);
\path (-2,2.5) coordinate (v00);
\path (-4,2.5) coordinate (v10);

\path (-7,2.5) coordinate (v1);
\path (-9,2.5) coordinate (v0);

\draw (v0) -- (v1);
\draw (v10) -- (v00);\draw (v00) -- (v01);

\draw (v010) -- (v000) -- (v001);
\draw (v000) -- (v100) -- (v101) -- (v001);
%
\draw (v1)   [fill=black]circle (\vr);
\draw (v0)  [fill=black] circle (\vr);

\draw (v01)  [fill=black] circle (\vr);
\draw (v00)  [fill=black] circle (\vr);
\draw (v10)  [fill=black] circle (\vr);

\draw (v001)  [fill=black] circle (\vr);
\draw (v000)  [fill=black] circle (\vr);
\draw (v010)  [fill=black] circle (\vr);
\draw (v101)  [fill=black] circle (\vr);
\draw (v100)  [fill=black] circle (\vr);
\draw[right] (v0)++(-0.1,-0.3) node {$0$};
\draw[right] (v1)++(-0.1,-0.3) node {$1$};
\draw[right] (v10)++(0.1,-0.3) node {$10$};
\draw[right] (v00)++(-0.1,-0.3) node {$00$};
\draw[right] (v01)++(-0.1,-0.3) node {$01$};
\draw[left] (v100)++(-0.0,0) node {$100$};
\draw[right] (v101)++(0.0,0.1) node {$110$};
\draw[right] (v010)++(0.1,-0.3) node {$001$};
\draw[right] (v000)++(-0.1,-0.3) node {$000$};
\draw[right] (v001)++(-0.1,-0.3) node {$010$};
\end{scope}
\begin{scope}[yshift = 0cm, xshift = 0cm]
\path (5,0) coordinate (v001001);
\path (2.5,0) coordinate (v001000);
\path (0,0) coordinate (v001010);
\path (5,2.5) coordinate (v000001);
\path (2.5,2.5) coordinate (v000000);
\path (0,2.5) coordinate (v000010);
\path (5,5) coordinate (v000101);
\path (2.5,5) coordinate (v000100);
\path (7.0,4) coordinate (v010001);
\path (4.5,4) coordinate (v010000);
\path (2.0,4) coordinate (v010010);
\path (7.0,1.5) coordinate (v010101);
\path (4.5,6.5) coordinate (v010100);
\draw (v001001) -- (v001000) -- (v001010) -- (v000010) -- (v000000) -- (v000001) -- (v001001);
\draw (v001000) -- (v000000) -- (v000100) -- (v000101) -- (v000101) -- (v000001);
\draw (v010100) -- (v000100);

\draw (v000001) -- (v010001) -- (v010000);
\draw (v010000) -- (v010100);

\draw (v010100) -- (v000100);
\draw (v010000) -- (v010010);
\draw (v010010) -- (v000010);
\draw (v001001) -- (v010101) -- (v010001);
\draw (v000000) -- (v010000);
\draw (v001001)  [fill=white] circle (\vr);
\draw (v001000)  [fill=white] circle (\vr);
\draw (v001010)  [fill=white] circle (\vr);
\draw (v000001)  [fill=white] circle (\vr);
\draw (v000000)  [fill=white] circle (\vr);
\draw (v000010)  [fill=white] circle (\vr);
\draw (v010001)  [fill=white] circle (\vr);
\draw (v010000)  [fill=white] circle (\vr);
\draw (v010010)  [fill=white] circle (\vr);
\draw (v000101)  [fill=white] circle (\vr);
\draw (v000100)  [fill=white] circle (\vr);
\draw (v010101)  [fill=white] circle (\vr);
\draw (v010100)  [fill=white] circle (\vr);
\end{scope}

\begin{scope}[xshift = -8cm, yshift = 0 cm]
\path (5,0) coordinate (v101001);
\path (2.5,0) coordinate (v101000);
\path (0,0) coordinate (v101010);
\path (5,2.5) coordinate (v100001);
\path (2.5,2.5) coordinate (v100000);
\path (0,2.5) coordinate (v100010);
\path (5,5) coordinate (v100101);
\path (2.5,5) coordinate (v100100);
\draw(v101001) -- (v101000);
\draw(v101000) -- (v101010);
\draw(v101010) -- (v100010);

\draw (v100010) -- (v100000) -- (v100001);
\draw (v100001) -- (v101001);

\draw (v101000) -- (v100000);
\draw (v100000) -- (v100100) -- (v100101) -- (v100001);

\draw (v101001)  [fill=black] circle (\vr);
\draw (v101000)  [fill=black] circle (\vr);
\draw (v101010)  [fill=black] circle (\vr);
\draw (v100001)  [fill=black] circle (\vr);
\draw (v100000)  [fill=black] circle (\vr);
\draw (v100010)  [fill=black] circle (\vr);
\draw (v100101)  [fill=black] circle (\vr);
\draw (v100100)  [fill=black] circle (\vr);

\draw[left] (v100100)++(-0.0,0) node {$0010$};
\draw[right] (v100101)++(0.0,0.1) node {$0110$};
\draw[right] (v100010)++(0.1,-0.3) node {$0001$};
\draw[right] (v100000)++(-0.1,-0.3) node {$0000$};
\draw[right] (v100001)++(-0.1,-0.3) node {$0100$};
\draw[below] (v101000)++(0.0,-0.1) node {$1000$};
\draw[below] (v101010)++(0.0,-0.1) node {$1001$};
\draw[below] (v101001)++(0.0,-0.1) node {$1100$};

\draw[right] (v010100)++(0.0,0.0) node {$10001$};
\draw[right] (v000101)++(0.0,0.1) node {$01001$};
\draw[left] (v000100)++(0.0,-0.0) node {$00001$};
\draw[left] (v010000)++(0.0,0.25) node {$10000$};
\draw[left] (v010010)++(0.0,0.25) node {$10010$};
\draw[right] (v010001)++(0,0.25) node {$11000$};
\draw[right] (v000001)++(-0.1,-0.3) node {$01000$};
\draw[right] (v000010)++(-0.1,-0.3) node {$00010$};
\draw[right] (v000000)++(-0.1,-0.3) node {$00000$};
\draw[right] (v010101)++(0.0,0.0) node {$11100$};
\draw[below] (v001000)++(0,-0.1) node {$00100$};
\draw[below] (v001001)++(0,-0.1) node {$01100$};
\draw[below] (v001010)++(0.0,-0.1) node {$00110$};

\end{scope}
\draw (v001001)  [fill=black] circle (\vr);
\draw (v001000)  [fill=black] circle (\vr);
\draw (v001010)  [fill=black] circle (\vr);
\draw (v000001)  [fill=black] circle (\vr);
\draw (v000000)  [fill=black] circle (\vr);
\draw (v000010)  [fill=black] circle (\vr);
\draw (v010001)  [fill=black] circle (\vr);
\draw (v010000)  [fill=black] circle (\vr);
\draw (v010010)  [fill=black] circle (\vr);
\draw (v000101)  [fill=black] circle (\vr);
\draw (v000100)  [fill=black] circle (\vr);
\draw (v010101)  [fill=black] circle (\vr);
\draw (v010100)  [fill=black] circle (\vr);
\draw (v010100)  [fill=black] circle (\vr);
\end{tikzpicture}
\caption{The graphs $\R_n$, for $n \in [5]$}
\label{fig:R_5}
\end{center}

\end{figure}
We can associate each hypercube $Q_k$ in $\R_n$ with its top vertex. The calculation of the generating function in proof of Proposition~8.2 is obtained by summing for all vertex $v \in V(\R_n)$ the contribution coming from hypercubes with top vertex $v$.  
It is then written erroneously that a vertex $v$ with down-degree $r$ contributes
\begin{equation*}
\sum_{k=0}^r {r \choose k } q^{r-k} x^k = 
(q+x)^r\,
\end{equation*}
while this contribution will also depend on the Hamming weight $w$ of $v$ since the distance to $0^n $  of a $Q_k$ with top vertex $v$ is $w-k$ and not $r-k$.
Fortunately by refining the approach of the authors of~\cite{EI-2021a} it is possible to to determine the correct generating function.

Let us first note that although $\R_n (n\ge7)$ is not an isometric graph of $Q_n$~(\cite[Proposition~8.1]{EI-2021a}) for any vertex $v$ of $\R_n$ with Hamming weight $w$ 
\begin{equation*}
d_{\R_n}(v,0^n)=d_{Q_n}(v,0^n)=w.
\end{equation*}
Indeed since $\R_n$ is a subgraph of $Q_n$, $d_{\R_n}(v,0^n)\geq w$. On the other hand, since by changing the first $1$ in a run of $1$s into a $0$ we obtain a vertex of $\R_n$ with weight $w-1$, it is easy to create a $v,0^n$-path whose length is $w$.

We define the {\em down-degree co-weight polynomial} of $\R_n$ as the bivariate polynomial
\begin{equation*}
 DCW_{\R_n}(d,z)=\sum_{v\in V(\R_n)}d^{\deg_{\rm down} (v)}z^{w(v)-{\deg_{\rm down} (v)}}. 
\end{equation*}
where $w(v)$ and $\deg_{\rm down} (v)$ are, respectively, the weight and the down-degree of $v$.

The coefficient of the monomial $d^az^b $ in $DCW_{\R_n}(d,z)$ is thus the number of vertices with down-degree $a$ and Hamming weight $a+b$.

We proceed like the determination of the generating function of the down-degree enumerator polynomials for $\R_n${~\cite[Proposition~7.1]{EI-2021a}}.
\begin{proposition}
\label{prop:RdownW}
	The generating function of down-degree co-weight enumerator polynomials of Fibonacci-run graphs is
	\begin{equation*}
\sum_{n \geq 0} 
	t^n \sum_{v\in V(\R_n)}d^{{\rm deg}_{\rm down}(v)}z^{{\rm w}(v)-{{\rm deg}_{\rm down}(v)}} = 
\frac{1+dt+(d-z)t^2+d(d-z)t^3+d(d-z)t^4}{1-t-zt^2-(d-z) t^3 -d(d-z)t^5} \,.
\end{equation*}
\end{proposition}
\begin{proof}
Let us recall that the strings of $\bigcup_{n\geq 0} V(\R_n)$ are obtained from the strings of length at least 2 in the set $\M$ of  run-constrained strings after suppressing the tail $00$. $\M$ is generated as a monoid by  the infinite alphabet $R= \{0,100,11000,1110000,\ldots\}$.
 Associate with a string in $\M$ of length $n$ the monomial $d^rz^{w-r}t^n $ where $w$ is the number of $1$s and $r$ the number of $1$s that can be switched individually to $0$ while remaining in $\M$.

For example consider $s= 01001110000$. Since $00001110000$,$01000110000$ and $01001100000$ are in $\M$ but not
$01001010000$ we have  $r=3$ and the monomial associated to $s$ is $d^3zt^{11}$.

 Assume $s=a_1a_2\ldots a_k$ where the $a_i$s are in $R$. Then a 1 belongs to a unique $a_i$ like the run of $1$s of which it is part. Furthermore it can be switched to 0 in $s$ if and only if it is the first or the last of its run thus if and only if it can be switched to $0$ in $a_i$. Thus $s \rightarrow d^r$ is concatenation multiplicative over $\M$. Since $w \geq r$ and $s\rightarrow z^w$ is concatenation multiplicative $s\rightarrow z^{w-r}$ satisfies the same property. By product this is also true for $s \rightarrow d^rz^{w-r}t^n$.

We can thus apply Proposition~\ref{pro:mono}.

Consider the contributions of the strings from the alphabet $R$.
	
The string $0$ contributes $t$, $100$ contributes 
	$dt^3$, but longer strings $1^k 0^{k+1}k,k \geq2$ contribute $d^2z^{k-2}t^{2k+1}$, 
since only the first and the last $1$ appearing can be switched to $0$. 
The strings from $R$ give 
	\begin{equation*}
	t + dt^3 + d^2t^5 + d^2zt^7 + \cdots  = t + dt^3 + \frac{d^2 t^5}{1-zt^2} \,.
\end{equation*}
	Therefore, using Proposition~\ref{pro:mono}, the total contribution of the elements of $\M=\{\lambda,0\}\cup(\bigcup_{n\geq 0}{V(\R_n)}00$ is 
	\begin{equation*}
	\frac{1}{
		1 -t - dt^3 - \frac{d^2 t^5}{1-zt^2}} \,.
\end{equation*}
	From this we need to subtract the terms $1, t,$ which correspond to the null-word and  $0$  respectively and we divide by $t^2$ to shorten the length by 2. We obtain
\begin{equation*}
\frac{1-zt^2}{t^2(1-t-zt^2-(d-z) t^3 -d(d-z)t^5) }-\frac{1+t}{t^2}=\frac{1+dt+(d-z)t^2+d(d-z)t^3+d(d-z)t^4}{1-t-zt^2-(d-z) t^3 -d(d-z)t^5}  \,.
\end{equation*}

	For a string in $V(\R_n)$, $w$ and $r$ can be interpreted as the Hamming weight and down-degree, respectively. Therefore  we obtain the proposition.
	
\end{proof}\qed

In the distance cube polynomial $D_{\R_n}(x,q)$ to each hypercube $Q_k$ in $\R_n$
we associate the monomial $q^d x^k$, where $d$ is the distance of the $Q_k$ to $0^n$. 
\begin{proposition}
\label{prop:gCubes}
The generating function 
 of $(D_{\R_n}(x,q))_{n\geq 0}$ is given by
	{\small
		\begin{equation*}
		\sum_{n\geq 0} D_{\R_n}(x,q) t^n =\frac{1+(q+x)t+ xt^2+x(q+x)t^3+x(q+x)t^4}{1-t-qt^2-xt^3-x(q+x)t^5}  \, .
		\end{equation*}

	}
\end{proposition}
\begin{proof}

Let $v$ be a vertex of $\R_n$ with weight $w$ and down-degree $r$. Select $k$ $1$s in its string representation among the $r$ that can be replaced with a $0$. By this way we obtain a copy of $Q_k$ in $\R_n$ with top vertex $v$
and conversely to any $Q_k$ with top vertex $v$ we can associate such a unique choice. The distance of this $Q_k$ to $0^n$ is $w-k$. The contribution to $D_{\R_n}(x,q)$ of the hypercubes with top vertex $v$ is thus 
\begin{equation*}
\sum_{k=0}^r {r \choose k } q^{w-k} x^k = q^{w-r}\sum_{k=0}^r {r \choose k } q^{r-k} x^k = q^{w-r}(q+x)^r\, .
\end{equation*}

By uniqueness of the top vertex of an hypercube we have thus
\begin{equation*}
D_{\R_n}(x,q)= \sum_{v\in V(\R_n)}q^{w(v)-{\deg_{\rm down} (v)}}(q+x)^{\deg_{\rm down} (v)}. 
\end{equation*}

Since
\begin{equation*}
 DCW_{\R_n}(d,z)=\sum_{v\in V(\R_n)}d^{\deg_{\rm down} (v)}z^{w(v)-{\deg_{\rm down} (v)}} 
\end{equation*}
 we deduce
\begin{equation}
D_{\R_n}(x,q)=DCW_{\R_n}(q+x,q).
\end{equation}

Replacing $d$ by $q+x$ and $z$ by $q$ in the 
generating function in Proposition~\ref{prop:RdownW} gives the expression.
\end{proof}\qed
 
The values of $D_{\R_n}(x,q)$ given for $ n \in [6] $ by   Proposition~\ref{prop:gCubes}  are
\begin{align*}
	&1+q+x,\\
	&1+2q + 2x,\\
	&1+3q+q^2 + (3 +2q) x + x^2,\\
	&1+4q+3q^2+ (4+6q)x + 3 x^2,\\
	&1+5q+6q^2+q^3+(5+12q+2q^2)x + (6+q) x^2, \\
	& 1 + 6q + 10q^2 + 4q^3 + (6+20q+10q^2)x +(10+8q)x^2 + 2 x^3.
\end{align*}

Note that after contacting the authors of~\cite{EI-2021a} there is a typographical error in this article concerning the generating function of the up-degree~\cite[Proposition 7.2]{EI-2021a}. This one obtained as a specialization of the up-down degree polynomial~\cite[Theorem 5.1]{EI-2021b} is in fact
\begin{equation*}
\sum_{n\geq1}t^n\sum_{v\in V(\R_n)}u^{\deg_{\rm up(v)}}=\frac{t(1+u-(u-2)t-2ut^2+(u-2)t^3-(u-1)t^5-(u-1)t^6)}{1-ut-2t^2+(2u-1)t^3+t^4-(u-1)t^5+(u-1)t^7}\,.
\end{equation*}

The first terms indicated in~\cite{EI-2021a} are correct because calculated by expansion of the expression without typo.

\section{Some complements about Lucas-run graphs}\label{sec:LR}
\bigskip

The specialization $q=1$ in the expressions of $D_{\R_n}(x,q)$ gives the cube polynomial. Note that for $q=1$ we have $(q+x)^r = q^{w-r}(q+x)^r$ and therefore the expression for $C_{\R_n}(x)$ that we  deduce from the erroneous expression of  $D_{\R_n}(x,q)$ in ~\cite[Proposition~8.2]{EI-2021a} already gave the correct value for the cube polynomial.

The  first values of $C_{\R_n^l}(x)$ obtained by direct  inspection  of figure \ref{fig:R_5c} suggest that, for $n\geq2$, $C_{\R_n^l}(x)=2C_{\R_{n-1}}(x)-C_{\R_{n-2}}(x)$~\cite[Conjecture~5.3]{W-2024a}. We will prove a stronger result: the equality occurs for the distance cube polynomial.
\begin{figure}[!ht]
\begin{center}
\begin{tikzpicture}[scale=0.65,style=thick]
\tikzstyle{every node}=[draw=none,fill=none]
\def\vr{3pt} 
\begin{scope}[xshift =0, yshift = 6cm]

\path (7,2.5) coordinate (v001);
\path (5,2.5) coordinate (v000);
\path (3,2.5) coordinate (v010);
\path (7,4.5) coordinate (v101);

\path (0,2.5) coordinate (v01);
\path (-2,2.5) coordinate (v00);
\path (-4,2.5) coordinate (v10);

\path (-9,2.5) coordinate (v0);

\draw (v10) -- (v00);\draw (v00) -- (v01);

\draw (v010) -- (v000) -- (v001);
\draw (v001) -- (v000) -- (v100);
\draw (v010) -- (v000);

%
\draw (v0)  [fill=black] circle (\vr);

\draw (v01)  [fill=black] circle (\vr);
\draw (v00)  [fill=black] circle (\vr);
\draw (v10)  [fill=black] circle (\vr);

\draw (v001)  [fill=black] circle (\vr);
\draw (v000)  [fill=black] circle (\vr);
\draw (v010)  [fill=black] circle (\vr);
\draw (v100)  [fill=black] circle (\vr);
\draw[right] (v0)++(-0.1,-0.3) node {$0$};
\draw[right] (v10)++(0.1,-0.3) node {$10$};
\draw[right] (v00)++(-0.1,-0.3) node {$00$};
\draw[right] (v01)++(-0.1,-0.3) node {$01$};
\draw[left] (v100)++(-0.0,0) node {$100$};
\draw[right] (v010)++(0.1,-0.3) node {$001$};
\draw[right] (v000)++(-0.1,-0.3) node {$000$};
\draw[right] (v001)++(-0.1,-0.3) node {$010$};
\end{scope}
\begin{scope}[yshift = 0cm, xshift = 0cm]
\path (5,0) coordinate (v001001);
\path (2.5,0) coordinate (v001000);
\path (0,0) coordinate (v001010);
\path (5,2.5) coordinate (v000001);
\path (2.5,2.5) coordinate (v000000);
\path (0,2.5) coordinate (v000010);
\path (5,5) coordinate (v000101);
\path (2.5,5) coordinate (v000100);
\path (7.0,4) coordinate (v010001);
\path (4.5,4) coordinate (v010000);
\path (2.0,4) coordinate (v010010);
\path (7.0,1.5) coordinate (v010101);
\path (4.5,6.5) coordinate (v010100);
\draw (v001001) -- (v001000) -- (v001010) -- (v000010) -- (v000000) -- (v000001) -- (v001001);
\draw (v001000) -- (v000000) -- (v000100) -- (v000101) -- (v000101) -- (v000001);

\draw (v000001) -- (v010001) -- (v010000);

\draw (v010000) -- (v010010);
\draw (v010010) -- (v000010);
\draw (v000000) -- (v010000);
\draw (v001000)  [fill=white] circle (\vr);
\draw (v001010)  [fill=white] circle (\vr);
\draw (v000001)  [fill=white] circle (\vr);
\draw (v000000)  [fill=white] circle (\vr);
\draw (v000010)  [fill=white] circle (\vr);
\draw (v010001)  [fill=white] circle (\vr);
\draw (v010000)  [fill=white] circle (\vr);
\draw (v010010)  [fill=white] circle (\vr);
\draw (v000101)  [fill=white] circle (\vr);
\draw (v000100)  [fill=white] circle (\vr);
\end{scope}

\begin{scope}[xshift = -8cm, yshift = 0 cm]
\path (5,0) coordinate (v101001);
\path (2.5,0) coordinate (v101000);
\path (0,0) coordinate (v101010);
\path (5,2.5) coordinate (v100001);
\path (2.5,2.5) coordinate (v100000);
\path (0,2.5) coordinate (v100010);
\path (5,5) coordinate (v100101);
\path (2.5,5) coordinate (v100100);
\draw(v101001) -- (v101000);

\draw (v100010) -- (v100000) -- (v100001);
\draw (v100001) -- (v101001);

\draw (v101000) -- (v100000);
\draw (v100000) -- (v100100) -- (v100101) -- (v100001);

\draw (v101001)  [fill=black] circle (\vr);
\draw (v101000)  [fill=black] circle (\vr);
\draw (v100001)  [fill=black] circle (\vr);
\draw (v100000)  [fill=black] circle (\vr);
\draw (v100010)  [fill=black] circle (\vr);
\draw (v100101)  [fill=black] circle (\vr);
\draw (v100100)  [fill=black] circle (\vr);

\draw[left] (v100100)++(-0.0,0) node {$0010$};
\draw[right] (v100101)++(0.0,0.1) node {$0110$};
\draw[right] (v100010)++(0.1,-0.3) node {$0001$};
\draw[right] (v100000)++(-0.1,-0.3) node {$0000$};
\draw[right] (v100001)++(-0.1,-0.3) node {$0100$};
\draw[below] (v101000)++(0.0,-0.1) node {$1000$};
\draw[below] (v101001)++(0.0,-0.1) node {$1100$};

\draw[right] (v000101)++(0.0,0.1) node {$01001$};
\draw[left] (v000100)++(0.0,-0.0) node {$00001$};
\draw[left] (v010000)++(0.0,0.25) node {$10000$};
\draw[left] (v010010)++(0.0,0.25) node {$10010$};
\draw[right] (v010001)++(0,0.25) node {$11000$};
\draw[right] (v000001)++(-0.1,-0.3) node {$01000$};
\draw[right] (v000010)++(-0.1,-0.3) node {$00010$};
\draw[right] (v000000)++(-0.1,-0.3) node {$00000$};
\draw[below] (v001000)++(0,-0.1) node {$00100$};
\draw[below] (v001001)++(0,-0.1) node {$01100$};
\draw[below] (v001010)++(0.0,-0.1) node {$00110$};

\end{scope}
\draw (v001001)  [fill=black] circle (\vr);
\draw (v001000)  [fill=black] circle (\vr);
\draw (v001010)  [fill=black] circle (\vr);
\draw (v000001)  [fill=black] circle (\vr);
\draw (v000000)  [fill=black] circle (\vr);
\draw (v000010)  [fill=black] circle (\vr);
\draw (v010001)  [fill=black] circle (\vr);
\draw (v010000)  [fill=black] circle (\vr);
\draw (v010010)  [fill=black] circle (\vr);
\draw (v000101)  [fill=black] circle (\vr);
\draw (v000100)  [fill=black] circle (\vr);
\end{tikzpicture}
\caption{The graphs $\R_n^l$, for $n \in [5]$}
\label{fig:R_5c}
\end{center}
\end{figure}
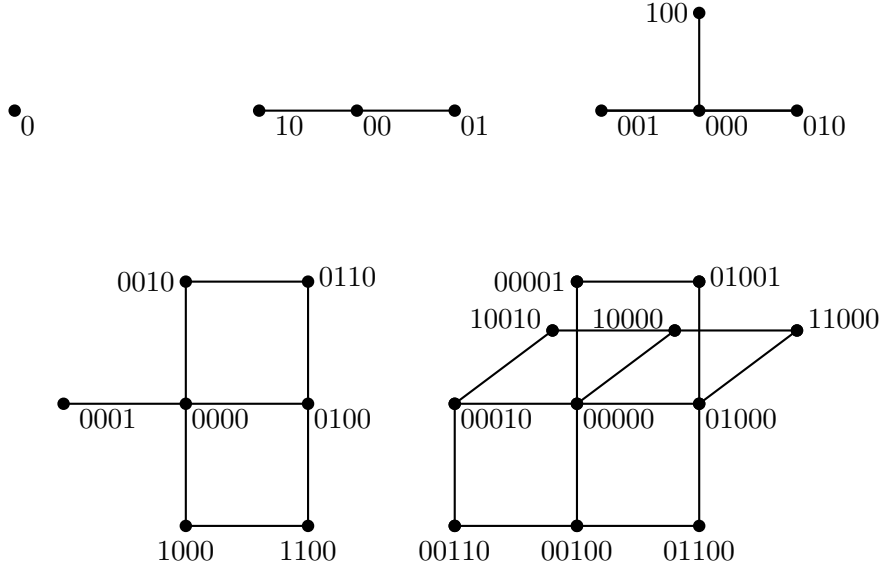

\begin{lemma}
\label{lem:begin1}
A string $s$ in $V(\R_n)$ that starts with $0$ belongs to $V(\R_n^l)$.
		A string $s$ in $V(\R_n)$ that start with $1$ belongs to $V(\R_n^l)$ if and only if $s00$ belongs to $\M0$.
\end{lemma}
\begin{proof}
Write $s00=a_1a_2\ldots a_k$ with $a_i\in R$ for $i\in[k]$ and let  $p$ be the length of the last run of $1$s. This last run of $1$s belongs to some $a_j$ and is followed in $a_j$ by exactly $p+1$ $0$s.

 If $s$ starts with $0$. Since $s00$ ends with at least $p+1$ $0$s it is clear that $s0$ is a circular-run-constrained string.

 Assume now that $s$ starts with $1$ thus $a_1\neq0$. Then  $s0$ is a circular-run-constrained string if and only if $s0$ ends with at least $p+1$ $0$s. This is satisfied if and only if $s00$ ends with at least $p+2$ $0$s. This is not the case if the last run of $1$ belongs to $a_k$,  and thus this is  equivalent to $a_k=0$. Since $a_1\neq0$ we have $k\geq 2$, and $a_k=0$ if and only $s00$ is the concatenation of $a_1s_2\ldots a_{k-1}$ with $0$.
\end{proof}\qed

Let $\M^l$ be the set of  run-constrained strings that came from a Lucas-run graph, in other words  $\M^l=\bigcup_{n\geq 0}{V(\R_n^l)00}\,\subset\,\M$.

\begin{corollary}
\label{corr:OM}\label{eq:fib-lucas2}
\begin{equation}\label{eq:Ml}
		\M^l=(0\M \setminus \{0\})\cup ((\M0 \setminus \{0\}) \setminus 0\M0) \, .
\end{equation}
\end{corollary}
\begin{proof}
The strings of  $0\M$ are the strings of $\M$ that start with $0$.

Therefore 
those  of $0\M \setminus \{0\}$ are the strings of $\bigcup_{n\geq 0}{V(\R_n)00}$ that start with $0$. By Lemma~\ref{lem:begin1} they are the strings of $ \M^l$ that start $0$. 

The strings of $(\M0 \setminus \{0\}) \setminus 0\M0 $ are the strings of $\M0$ of length at least 2 that do not start with  $0$ thus by the same Lemma they are  the strings of $ \M^l$ that start $1$. 
\end{proof}\qed

We can now prove our main result in this section.
\begin{theorem}
\label{conjjecturegen}
$D_{\R_n^l}(x,q)= 2D_{\R_{n-1}}(x,q)-D_{R_{n-2}}(x,q)$ holds for any $n\geq2$.
\end{theorem}
\begin{proof}
As in the proof of Proposition~\ref{prop:RdownW} associate with a string $s$ in $\M$ of length $n$ the monomial $\theta_s=d^rz^{w-r}t^n $ where $w$ is the number of $1$s and $r$ the number of $1$s that can be switched individually to $0$ while remaining in $\M$ and let us recall that  $s\rightarrow \theta_s$  is concatenation multiplicative over $\M$.

Note that $\sum_{s\in \M0}{\theta_s}=t\sum_{s\in \M}{\theta_s}$ and $\theta_{0s0}=t^2\theta_{s}$. From the equality~\ref{eq:Ml} we obtain
\begin{equation}\label{eq:theta}
\sum_{s\in \M^l}{\theta_s}=t(\sum_{s\in \M}{\theta_s}-1)+t(\sum_{s\in \M}{\theta_s}-1)-t^2\sum_{s\in \M}{\theta_s}\, .
\end{equation}

Since $\M=\{\lambda,0\}\cup(\bigcup_{n\geq 0}{V(\R_n)}00)$ the coefficient of $t^{m+2}$ in $\sum_{s\in \M}{\theta_s}$ is the coefficient of $t^{m+2}$ in $\sum_{s\in V(\R_m)}{\theta_{s00}}$. Because $r(s00)=r(s)$ and $w(s00)=w(s)$ it is the coefficient of $t^{m}$ in  $\sum_{s\in V(\R_m)}{\theta_{s}}$ and thus this coefficient is  $DCW_{\R_m}(d,z)$ the down-degree co-weight enumerator polynomial of $R_m$.

Therefore the coefficient of $t^{n+2}$ in the right side of~\ref{eq:theta} is 
\begin{equation*}
2DCW_{\R_{n-1}}(d,z)-DCW_{\R_{n-2}}(d,z)\, .
\end{equation*}

Let us now prove that the coefficient of $t^{n+2}$ in the left side is $DCW_{\R_{n}^l}(d,z)$. Since $w(s00)=w(s)$ and $r(s00)=r(s)$ we have $\sum_{s\in V(\R_n^l)00}{\theta_s}=t^2\sum_{s\in V(\R_n^l)}{\theta_s}$
and we are thus looking for the coefficient of $t^n$ in $\sum_{s\in V(\R_n^l)}{\theta_s}$.

Let $s\in V(\R_n^l)$ and assume that after switching the  $i$th $1$ in $s$ we obtain $s'\in\M$.

 If this $1$ doesn't belong to the last run of $1$s in $s$ it is immediate that $s'\in V(\R_n^l)$.

 Assume now this $1$ belongs to the last run. Then the length of the last run of $1$s is smaller in $s'$ than in $s$.  The length of the last run of $0$s is the same or increase by one if that $1$ is the last in the run. Thus in  both cases we have $s'\in V(\R_n^l)$. Therefore $r(s)$ is the down-degree of the vertex $s$ in $\R_n^l$.

 The coefficient of $t^{n+2}$ in the left side of~\ref{eq:theta} is thus $DCW_{\R_{n}^l}(d,z)$ and we obtain, for $n\geq 2$, the equality
\begin{equation}\label{eq:DCWLF}
DCW_{\R_{n}^l}(d,z)=2DCW_{\R_{n-1}}(d,z)-DCW_{\R_{n-2}}(d,z)\, . 
\end{equation}

The rest of the proof is similar to that of the case of Fibonacci-run graphs in Proposition~\ref{prop:gCubes}. Let $v$ be a vertex of $\R_n$ with weight $w$ and down-degree $r$. The contribution to $D_{\R_n^l}(x,q)$ of the hypercubes with top vertex $v$ is  $
\sum_{k=0}^r {r \choose k } q^{w-k} x^k =q^{w-r}(q+x)^r.$
Therefore 
\begin{equation*}
D_{\R_n^l}(x,q)= \sum_{v\in V(\R_n^l)}q^{w(v)-{\deg_{\rm down} (v)}}(q+x)^{\deg_{\rm down} (v)}= DCW_{\R_n^l}(q+x,q).
\end{equation*}
Since $D_{\R_n}(x,q)= DCW_{\R_n}(q+x,q)$ the conclusion follows from Equation \ref{eq:DCWLF}.

\end{proof}\qed
\begin{corollary}
\label{prop:lCubes} 
The generating function 
 of $(D_{\R_n^l}(x,q))_{n\geq 0}$ is given by
	{\small
		\begin{equation*}
		\sum_{n\geq 0} D_{\R_n}(x,q) t^n=\frac{1+(q+2x)t^2+2x(q+x)t^4}{1-t-qt^2-xt^3-x(q+x)t^5}    \, .
		\end{equation*}

	}
\end{corollary}
\begin{proof}
For $n\geq2$ we deduce  from  Theorem~\ref{conjjecturegen} the equality 
\begin{equation*}
 t^nD_{\R_n^l}(x,q)= 2t^nD_{\R_{n-1}}(x,q)-t^nD_{R_{n-2}}(x,q)
\end{equation*}
thus by summation over $\{n\geq2\}$
\begin{equation*}
	\sum_{n\geq 0} D_{\R_n^l}(x,q) t^n -1-t=2t(\sum_{n\geq 0} D_{\R_n}(x,q)-1)-t^2(\sum_{n\geq 0} D_{\R_n}(x,q))
\end{equation*}
and therefore 
\begin{equation*}
\sum_{n\geq 0} D_{\R_n^l}(x,q) t^n= 1-t+  (2t-t^2)(\sum_{n\geq 0} D_{\R_n}(x,q))\, .
\end{equation*}

\end{proof}\qed
\section{Concluding remarks}
\label{sec:Conclusion}
The daisy cubes are an important family of subgraphs of hypercubes introduced in~\cite{KM-2019a}.  For any vertex of a daisy cube $G$  the down-degree is equal to the Hamming weight. The down-degree co-weight polynomial of a daisy cube $G$ therefore reduces to the weight enumerator polynomial.
\begin{equation*}
DCW_{G}(d,z)= \sum_{v\in V(G)}d^{\deg_{\rm down} (v)}z^{w(v)-{\deg_{\rm down} (v)}}=\sum_{v\in V(G)}d^{w (v)}=W_{G}(d) \, .
\end{equation*}

Fibonacci cubes, Lucas cubes, Alternate Lucas-cube~\cite{ESS-2021e} and  Pell graph~\cite{M-2019}\cite[Theorem 9.68 for a proof]{EKM-2023} are examples of daisy cubes. 

It is proved in~\cite[Corollary 2.6]{KM-2019a}, using the inclusion-exclusion principle, that for a daisy cube
\begin{equation*}
D_G(x,q)=C_G(x+q-1)=W_G(x+q)\,.
\end{equation*}
Thus daisy cubes,  Fibonacci-run graphs and Lucas-run graphs, satisfy the relation 
\begin{equation}\label{eqn:daisy}
D_G(x,q)=DCW_G(x+q,q)\,.
\end{equation}
A first open question is to determine the subgraphs of hypercubes such that equality~\ref{eqn:daisy} holds.

Note that for a daisy cube $D_G(x,-x)=C_G(-1)=1$ since $W_G(0)=1$. Another important class of graphs, related to hypercubes, is that of median graphs~\cite{M-1978}.  Soltan and 
Chepoi~\cite{SC-1987} and independently \v{S}krekovski~\cite{S-2001} proved that if $G$ is a median graph then $C_G(-1)=1$. This equality in particular generalizes the fact that $n(T) - m(T) = 1$ holds for any tree $T$. 
Fibonacci-run graphs and Lucas-run graphs are, in general, neither median nor daisy cubes but from equation~\ref{eqn:daisy}  satisfy also $D_G(x,-x)=1$ since $0^n$ is the only vertex with down-degree $0$. 

Setting $x=-1$  we obtain
\begin{equation*}
	 D_{\R_n}(-1,1)=C_{\R_n}(-1)=1   \, ,\,\,
D_{\R_n^l}(-1,1)=C_{\R_n^l}(-1)=1 \,.
\end{equation*}

A natural problem is thus to characterize the class of subgraphs $G$ of hypercubes for which $C_G(-1)=1$ holds.

The work done in Sections~\ref{sec:FR} and~\ref{sec:LR} on the monoid $\M$ of run-constrained strings can be done on the monoid $\Fib$ of extended Fibonacci strings. Although it does not produce new results, it is interesting to note that the generating function of $DCW_{\Gamma_n}(d,z)$, thus $\sum_{n\geq0}W_{\Gamma_n}(d)t^n$, can be determined as in Proposition~\ref{prop:RdownW}
by evaluating $\sum_{s}d^{w(s)}t^{l(s)}$ on the alphabet $F$.

The generating function of $(D_{\Gamma_n}(x,q))_{n\geq 0}$ is then determined like in Proposition~\ref{prop:gCubes}.

Lucas strings satisfy the equivalent of equation~\ref{eq:Ml}
\begin{equation*}
		\bigcup_{n\geq 0}{V(\Lambda_n)00}=(0\Fib \setminus \{0\})\cup ((\Fib0 \setminus \{0\}) \setminus 0\Fib0) \, .
\end{equation*}
and as a consequence
\begin{equation*}
D_{\Lambda_n}(x,q)= 2D_{\Gamma_{n-1}}(x,q)-D_{\Gamma_{n-2}}(x,q)
\end{equation*}
 holds for any $n\geq2$.

This latter equality has only been given explicitly, to my knowledge, in the particular case of the cube polynomial~\cite{KM-2012a} but could have been demonstrated independently from the generating functions given in~\cite{SE-2017a} and~\cite{SE-2018a}.
		%

\bibliographystyle{plain}  

\bibliography{fibo-bib}      

 \end{document}